\documentclass[12pt,reqno]{amsproc}
\usepackage[hyphens]{url} 
\numberwithin{equation}{section}

\makeatletter
\@namedef{subjclassname@2020}{%
\textup{2020} Mathematics Subject Classification}
\makeatother

\usepackage{amsmath,amssymb,color}
\usepackage{amsthm} 
\usepackage{multicol}
\usepackage{mathtools}
\usepackage[bookmarksnumbered,colorlinks]{hyperref}
\usepackage{enumitem}
\mathtoolsset{showonlyrefs}
\usepackage[pagewise,mathlines]{lineno}
\usepackage{enumitem}
\theoremstyle{plain} 
\newtheorem{theorem}{\indent\sc Theorem}[section]
\newtheorem{lemma}[theorem]{\indent\sc Lemma}
\newtheorem{corollary}[theorem]{\indent\sc Corollary}

\theoremstyle{definition} 

\title[Optimal lower estimate for the first eigenvalue of the $p$-Laplacian]{Optimal lower estimate for the first eigenvalue of the $p$-Laplacian in the Euclidean sphere}

\bigskip

\author[F.R. dos Santos and M.N. Soares]{F\'abio R. dos Santos$^{\ast}$ and Matheus N. Soares}

\address{
Departamento de Matem\'atica \\
Universidade Federal de Pernambuco \\
50.740-540 Recife, Pernambuco \\
Brazil}
\email{fabio.reis@ufpe.br}
\email{matheus.nsoares@ufpe.br}


\thanks{$^{\ast}$Corresponding author}

\begin{document}

\begin{abstract}
An integral inequality is derived for compact submanifolds (with or without boundary) in the unit sphere. This result leads to a characterization of spheres.\end{abstract}

\maketitle

\section{Introduction and statements of the main result}

Given a compact Riemannian manifold $M^{n}$ we define the $p$-Laplacian on $M^{n}$ as the second order quasilinear elliptic operator
\begin{equation}\label{eq:1.0}
\Delta_{p}f=-{\rm div}(|\nabla f|^{p-2}\nabla f),\quad 1<p<\infty.
\end{equation}
When $p=2$, it is the usual Laplacian. In a similar way, we can consider the eigenvalue problem of $\Delta_{p}$. We say that a real number $\lambda$ is a Dirichlet (or Neumann) eigenvalue if there exists a non-zero function $f$ satisfying the following equation with Dirichlet boundary condition $f=0$ on $\partial M$ (or Neumann boundary condition $u=\partial f/\partial\eta=0$ on $\partial M$):
\begin{equation}\label{eq:1.1}
\Delta_{p}f=\lambda|f|^{p-2}f\quad\mbox{in}\quad M^{n},
\end{equation}
where $\eta$ is the outward pointing normal unit vector field to $\partial M$.

According to \cite{Lindqvist:15}, these numbers $\lambda$ forms a non-increasing sequence such that there exists an isolated minimum eigenvalue called {\em the first eigenvalue} of the $p$-Laplacian (see also~\cite{le:06}). So, the first nontrivial Dirichlet and Neumann eigenvalue of $M^{n}$ are given by
\begin{equation*}\label{eq:1.2}
\mu_{1,p}(M)=\inf\left\{\dfrac{\int_{M}|\nabla f|^{p}dM}{\int_{M}|f|^{p}dM}\,;\,f\in W^{1,p}_{0}(M)\backslash\{0\}\right\},
\end{equation*}
and
\begin{equation}\label{eq:1.3}
\lambda_{1,p}(M)=\inf\left\{\dfrac{\int_{M}|\nabla f|^{p}dM}{\int_{M}|f|^{p}dM}\,;\,f\in W^{1,p}(M)\backslash\{0\}\,\,\mbox{and}\,\int_{M}|f|^{p-2}f\,dM=0\right\}.
\end{equation}
Since $\lambda_{1,p}(M)$ is obtained by imitating the closed case proof (see~\cite{Veron:91}), we will use the same notation to denote the first eigenvalue of $p$-laplacian of $M^{n}$ in the closed case.

In the geometric analysis theory, a natural way exists to relate the geometric properties of a Riemannian manifold and the $p$-Laplacian through its eigenvalues. The interaction between these objects and the geometry were approached in Matei \cite{Matei:00} which generalized many geometric intrinsic results involving the first eigenvalue of the $p$-Laplacian that, in principle, was proved only for the usual Laplacian. Among these results we highlight the generalization for $p>2$ of the Chern's comparison of geodesic balls, an extension of Faber-Krahn inequality and a generalization of Lichnerowicz-Obata theorem. 

On the other hand, in the context of isometric immersions in the unit sphere $\mathbb{S}^{N}$, Leung~\cite{Leung:83} established a sharp estimate between the eigenvalues of the usual Laplacian for minimal submanifolds of sphere through an integral inequality. In fact, Leung presented a lower bound for the square length of the second fundamental form $S$ via any eigenvalue of the Laplace-Beltrami operator. Besides that, he showed that spheres are obtained when the equality happens. Years later, Liu and Zhang \cite{Liu:07} proved a similar estimate for arbitrary submanifolds of sphere. More precisely, they proved
\begin{theorem}\label{thm:1.1}
Let $M^n$ be an $n$-dimensional compact orientable submanifold immersed in the standard Euclidean sphere $\mathbb{S}^{N}$. Let $f$ be an eigenfunction associated to $\lambda_{1,2}(M)$, then
\begin{equation}\label{eq:1.5}
\int_M\left(\frac{\sqrt{n-1}}{2}S-\dfrac{(n-1)(n-\lambda_{1,2}(M))}{n}\right)|\nabla f|^{2}dM\geq0,
\end{equation}
where $dM$ is the volume element on $M^{n}$. In particular, if $S$ is a constant, then 
\begin{equation}\label{eq:1.6}
S\geq\dfrac{2\sqrt{n-1}}{n}(n-\lambda_{1,2}(M)).
\end{equation}
\end{theorem}
We emphasize that, unlike the estimate obtained by Leung, the theorem above does not describe which geometric object achieves equality in integral inequality.

In a recent paper~\cite{Santos:23}, the authors considered compact minimal submanifolds in the unit sphere. They defined a divergence type operator and developed Bochner and Reilly formulas for it. As an application, they obtained integral inequalities involving the squared norm of the second fundamental form which extended the previous result due to Leung~\cite{Leung:83} to the first eigenvalue of the $p$-Laplacian and for manifolds with boundary. Here, our aim is to use the machinery developed in~\cite{Santos:23} in order to generalizes the Liu-Zhang result to the context of the $p$-Laplacian as well as to compact manifolds with nonempty boundary and possibility convex\footnote{We say that $\partial M$ is {\em convex} if the second fundamental form is negative semi-definite with respect to $\eta$ the outward-pointing unit normal vector.}. In other words, we prove

\begin{theorem}\label{thm:1.2}
Let $M^n$ be a compact (with possibly convex boundary) submanifold immersed in the standard Euclidean sphere $\mathbb{S}^{N}$. Let $f$ be an eigenfunction of the $p$-Laplacian of $M^n$ associated to $\lambda$ given by equation~\eqref{eq:1.1}. Then
\begin{equation}\label{eq:1.7}
\int_{M}\left(\frac{\sqrt{n-1}}{2}S-\dfrac{(n-1)(n-b(n,p)\lambda^{\frac{2}{p}})}{n}\right)|\nabla f|^{2p-2}dM \geq 0,
\end{equation}
for $p\geq2$, where
\begin{equation}\label{eq:1.8}
b(n,p)=\dfrac{n(p-1)^{2}-1}{(n-1)(p-1)^{\frac{2}{p}}}.
\end{equation} 
If the equality happen, then $p=2$. In this case, by assuming in addition that $M^{n}$ has parallel mean curvature vector field in $\mathbb{S}^{N}$, then $M^n$ is isometric to 
\begin{enumerate}[label=\alph*)]
\item[i.] the sphere $\mathbb{S}^2(\sqrt{q(q+1)/2})$, with $\lambda$ being the first eigenvalue and $N=2q$, for the closed case if $n=2$;
\item[ii.] the great sphere $\mathbb{S}^n(1)$, with $\lambda$ being the first eigenvalue, for the closed case if $n\geq 3$;
\item[iii.] the upper hemisphere $\mathbb{S}_{+}^n(1)$, with $\lambda$ being the first eigenvalue, corresponding to Dirichlet and Neumann problems.
\end{enumerate}
\end{theorem}

The proof of Theorem~\ref{thm:1.2} is given in Section~\ref{main result}. As a direct consequence of Theorem~\ref{thm:1.2} we get the following Leung's type estimate:
\begin{corollary}\label{cor:1.1}
In the context of Theorem~\ref{thm:1.2}, if $S$ is constant, then
\begin{equation}\label{eq:1.9}
S\geq \frac{2\sqrt{n-1}}{n}\left(n-b(n,p)\lambda_{1,p}^{2/p}(M)\right).
\end{equation}
\end{corollary}

\section{Some preliminaries and key lemmas}\label{preliminares}

Let $M^{n}$ be an $n$-dimensional connected submanifold immersed in a unit Euclidean sphere $\mathbb{S}^{N}$. Let $\{\omega_{B}\}$ be the corresponding dual coframe, and $\{\omega_{BC}\}$ the connection $1$-forms on $\mathbb{S}^{N}$. We choose a local field of orthonormal frame $\{e_{1},\ldots,e_{N}\}$ in $\mathbb{S}^{N}$, with dual coframe $\{\omega_{1},\ldots,\omega_{N}\}$, such that, at each point of $M^{n}$, $e_{1},\ldots,e_{n}$ are tangent to $M^{n}$ and $e_{n+1},\ldots,e_{N}$ are normal to $M^{n}$. We will use the following convection for indices
\begin{eqnarray*}\label{eq:2.1}
1\leq A,B,C,\ldots\leq N,\quad 1\leq i,j,k,\ldots\leq n\quad\mbox{and}\quad n+1\leq \alpha,\beta,\gamma,\ldots\leq N.
\end{eqnarray*}

With restricting on $M^{n}$, the second fundamental form $A$ and the curvature tensor $R$ of $M^{n}$ are given by
\begin{eqnarray*}\label{eq:2.2}
\omega_{i\alpha}=\sum_{j}h_{ij}^{\alpha}\omega_{j},\quad A=\sum_{i,j,\alpha}h_{ij}^{\alpha}\omega_{i}\otimes\omega_{j}\otimes e_{\alpha},
\end{eqnarray*}\label{eq:2.3}
\begin{eqnarray*}
d\omega_{ij}=\sum_{k}\omega_{ik}\wedge \omega_{kj}-\frac{1}{2}\sum_{k,l}R_{ijkl}\omega_{k}\wedge \omega_{l}.
\end{eqnarray*}
The Gauss equation is
\begin{eqnarray*}\label{eq:2.4}
R_{ijkl}=(\delta_{ik}\delta_{jl}-\delta_{il}\delta_{jk})+\sum_{\alpha}(h_{ik}^{\alpha}h_{jl}^{\alpha}-h_{il}^{\alpha}h_{jk}^{\alpha}).
\end{eqnarray*}
In particular, the components of the Ricci tensor $R_{ik}$ and the normalized scalar curvature $R$ are given, respectively, by
\begin{eqnarray}\label{eq:2.5}
R_{ik}=(n-1)\delta_{ik}+n\sum_{\alpha}\left(\sum_{j}h_{jj}^{\alpha}\right)h_{ik}^{\alpha}-\sum_{\alpha,j}h_{ij}^{\alpha}h_{jk}^{\alpha}
\end{eqnarray}
and
\begin{eqnarray}\label{eq:2.6}
R=\dfrac{1}{(n-1)}\sum_{i}R_{ii}.
\end{eqnarray}
From~\eqref{eq:2.5} and~\eqref{eq:2.6}, we get the following relation
\begin{eqnarray}\label{eq:2.7}
n(n-1)R=n(n-1)+n^{2}H^{2}-S,
\end{eqnarray}
where
\begin{equation}\label{eq:2.8}
S=\sum_{\alpha,i,j}(h_{ij}^{\alpha})^{2}\quad\mbox{and}\quad h=\dfrac{1}{n}\sum_{\alpha}\left(\sum_{k}h_{kk}^{\alpha}\right)e_{\alpha},
\end{equation}
denotes, respectively, the squared norm of the second fundamental form and the mean curvature vector field. Also, we define $H=|h|$ as the mean curvature function of $M^n$. In particular, we say that a submanifold $M^{n}$ immersed in the unit Euclidean sphere $\mathbb{S}^{N}$ has parallel mean curvature vector field if $h$ is parallel as a section of the normal bundle of $M^{n}$. It is clear that this condition implies in $H$ constant.

In order to proof our main results, we need of the following two results. The first one is a Ricci low estimate due~\cite[Main Theorem]{Leung:92}.
\begin{lemma}\label{lem:2.1}
Let $M^n$ be a submanifold of the Riemannian manifold $\mathbb{S}^{N}$. Let ${\rm Ric}$ denotes the function that assigns to each point of $M^{n}$ the minimum Ricci curvature. Then
\begin{equation}\label{eq:2.9}
{\rm Ric}\geq-\dfrac{n-1}{n}\left(S+\frac{n(n-2)}{\sqrt{n(n-1)}}H\sqrt{S-nH^2}-n-2nH^2\right).
\end{equation}
\end{lemma}

Before to present the second result, we will recall some facts about isometric immersions with nonempty boundary. Let us consider $\eta$ the outer unit normal field of $\partial M$. We define the shape operator $\mathcal{A}_{\eta}$ and the mean curvature function of $\partial M$ in $M^{n}$, respectively, by
\begin{equation}\label{eq:2.10}
\mathcal{A}_{\eta}(X)=-\nabla_{X}\eta\quad\mbox{and}\quad\mathcal{H}=\dfrac{1}{n-1}{\rm tr}(\mathcal{A}_{\eta}),
\end{equation}
for any $X\in\mathfrak{X}(\partial M)$. Let us denote by $\nabla^{\partial}$ and $\Delta^{\partial}$ the covariant derivative and the Laplacian operator on $\partial M$ with respect to the induced Riemannian metric. In this picture, the second key result is a suitable version of~\cite[Proposition $4.1$]{Santos:23} for our interest.

\begin{lemma}\label{lem:2.2}
Let $M^{n}$ be a compact manifold with boundary $\partial M$. If $f$ is an eigenfunction on $M^{n}$ of the $p$-Laplacian corresponding to a non-zero eigenvalue $\lambda$, then
\begin{equation}\label{eq:2.11}
\int_{M}{\rm Ric}(\nabla f,\nabla f)|\nabla f|^{2p-4}dM\leq\lambda^{\frac{2}{p}}c(n,p)\mathcal{R}(f,z,u)^{\frac{2p-2}{p}}+\int_{\partial M}|\nabla f|^{2p-4}Q(u,z)d\sigma,
\end{equation}
where
\begin{equation}\label{eq:2.12}
\mathcal{R}(f,z,u)=\|\nabla f\|_{2p-2}^{p}-\|f\|_{2p-2}^{2-p}\int_{\partial M}|f|^{p-2}|\nabla f|^{p-2}zu\,d\sigma
\end{equation}
and
\begin{equation}\label{eq:2.13}
\mathcal{Q}(u,z)=u\Delta^{\partial}z+(n-1)u^{2}\mathcal{H}+\langle\nabla^{\partial}z,\nabla^{\partial}u\rangle+\langle\mathcal{A}_{\eta}(\nabla^{\partial}z),\nabla^{\partial}z\rangle,
\end{equation}
with $z=f\big|_{\partial M}$, $u=\partial f/\partial\eta$ and $d\sigma$ denotes the Riemannian volume element on $\partial M$. Moreover, the equality holds in~\eqref{eq:2.11} if $p=2$.
\end{lemma}

\begin{proof}
From~\cite[Proposition 4.1]{Santos:23}, we have
\begin{equation}\label{eq:2.14}
\begin{split}
\int_{M}{\rm Ric}&(\nabla u,\nabla u)|\nabla u|^{2p-4}dM\\
&=\int_M\left((\Delta_p u)^2-|\nabla u|^{2p-4}|{\rm Hess}\,u|^2\right)dM+\int_{\partial M}|\nabla u|^{2 p-4}Q(u,z)d\sigma\\
&\quad-(p-2)\int_M|\nabla u|^{2p-6}\left((p-2)(\Delta_{\infty}u)^2|\nabla u|^{2}+2|{\rm Hess}\,u(\nabla u)|^2\right)dM,
\end{split}
\end{equation}
where $\mathcal{Q}(u,z)$ is given in~\eqref{eq:2.12} and $|\nabla f|^{2}\Delta_{\infty}f=\langle{\rm Hess}\,f(\nabla f),\nabla f\rangle$.

On the other hand, we note that holds the following algebraic inequality~\cite[Equation~$3.7$]{Santos:23}
\begin{equation}\label{eq:2.15}
|\nabla f|^{2p-4}|{\rm Hess}\,f|^2 \geq\dfrac{(\Delta_pf)^2}{n(p-1)^2},
\end{equation}
with equality holding if and only if
\begin{equation}\label{eq:2.16}
(p-2)|\nabla f|^{p-4}\langle{\rm Hess}\,f(\nabla f),X\rangle\nabla f+|\nabla f|^{p-2}{\rm Hess}\,f(X)=-\dfrac{1}{n}(\Delta_{p}f)I(X),
\end{equation}
for all $X\in\mathfrak{X}(M)$. Besides these, the Cauchy-Schwarz inequality guarantees
\begin{equation}\label{eq:2.17}
|\nabla f|^{4}(\Delta_{\infty}f)^{2}=\langle{\rm Hess}\,f(\nabla f),\nabla f\rangle^{2}\leq|\nabla f|^{2}|{\rm Hess}\,f(\nabla f)|^{2}
\end{equation}
and consequently
\begin{equation}\label{eq:2.18}
p(\Delta_{\infty}f)^2|\nabla f|^{2}\leq(p-2)(\Delta_{\infty}f)^2|\nabla f|^{2}+2|{\rm Hess}\,f(\nabla f)|^2.
\end{equation}
Hence, by inserting~\eqref{eq:2.15} and~\eqref{eq:2.18} in~\eqref{eq:2.14},
\begin{equation}\label{eq:2.19}
\begin{split}
\int_{M}{\rm Ric}(\nabla f,\nabla f)&|\nabla f|^{2p-4}dM\\
&\leq\left(\dfrac{n(p-1)^2-1}{n(p-1)^2}\right)\int_M\left(\Delta_pf\right)^2dM+\int_{\partial M}|\nabla f|^{2 p-4}Q(u,z)d\sigma\\
&\quad-p(p-2)\int_M|\nabla f|^{2p-4}(\Delta_{\infty}f)^2dM\\
&\leq\left(\dfrac{n(p-1)^2-1}{n(p-1)^2}\right)\int_M\left(\Delta_p f\right)^2dM+\int_{\partial M}|\nabla f|^{2p-4}Q(u,z)d\sigma,
\end{split}
\end{equation}
with equality occurring if and only if either $p=2$ or $\Delta_{\infty}f=0$.

On the other hand, by combining~\eqref{eq:1.1} with divergence's theorem and H{\"o}lder's inequality, we get
\begin{equation}\label{eq:2.21}
\begin{split}
\lambda \int_M |f|^{2p-2}dM&=(p-1)\int_M|f|^{p-2}|\nabla f|^{p}dM-\int_{\partial M}|f|^{p-2}|\nabla f|^{p-2}uzd\sigma\\
&\leq(p-1)\left(\int_M |\nabla f|^{2p-2}\right)^{\frac{p}{2p-2}}\left(\int_M | f|^{2p-2}\right)^{\frac{p-2}{2p-2}}- \int_{\partial M} |f|^{p-2}|\nabla f|^{p-2} uz d\sigma.
\end{split}
\end{equation}
By this, a direct computation gives
\begin{equation}\label{eq:2.22}
\int_M |f|^{2p-2}dM \leq\left(\frac{p-1}{\lambda}\right)^{\frac{2p-2}{p}}\mathcal{R}(f,z,u)^\frac{2p-2}{p},
\end{equation}
where $\mathcal{R}(f,z,u)$ is defined in~\eqref{eq:2.12}. Since $(\Delta_{p}f)^{2}=\lambda^{2}|f|^{2p-2}$, from~\eqref{eq:2.22}
\begin{equation}\label{eq:2.23}
\int_{M}(\Delta_{p}f)^{2}dM=\lambda^{2}\int_{M}|f|^{2p-2}dM\leq\lambda^{\frac{2}{p}}(p-1)^{\frac{2p-2}{p}}\mathcal{R}(f,z,u)^\frac{2p-2}{p}.
\end{equation}
Therefore, by inserting~\eqref{eq:2.23} in~\eqref{eq:2.21}, we obtain
\begin{equation}\label{eq:2.24}
\int_{M}{\rm Ric}(\nabla f,\nabla f)|\nabla f|^{2p-4}dM\leq\lambda^{\frac{2}{p}}c(n,p)\mathcal{R}(f,z,u)^{\frac{2p-2}{p}}+\int_{\partial M}|\nabla f|^{2p-4}Q(u,z)d\sigma,
\end{equation}
where
\begin{equation}\label{eq:2.25}
c(n,p)=\dfrac{(n-1)}{n}b(n,p).
\end{equation}
If the equality holds in~\eqref{eq:2.24}, then all the inequalities along the proof become equalities. In particular, the equality holds in~\eqref{eq:2.19} and~\eqref{eq:2.17} which implies that $p=2$ or ${\rm Hess}\,f(\nabla f)=0$. Thus, from~\eqref{eq:1.1} and~\eqref{eq:2.16}, if $\mathrm{Hess}\,f(\nabla f)=0$,
\begin{equation*}\label{eq:4.15}
0=(p-2)|\nabla f|^{p-4}\langle{\rm Hess}\,f(\nabla f),\nabla f\rangle\nabla f+|\nabla f|^{p-2}{\rm Hess}\,f(\nabla f)=-\dfrac{\lambda}{n}|f|^{p-2}fI(\nabla f).
\end{equation*}
Consequently, since $f\neq0$ and $\nabla f\neq0$ we obtain that $\lambda=0$, a contradiction. Therefore, if the equality in~\eqref{eq:2.24} holds, $p$ must be $2$.

\end{proof}

\section{Proof of Theorem~\ref{thm:1.2}}\label{main result}

From Lemma~\ref{lem:2.1}, we have that the Ricci curvature of $M^{n}$ satisfies:
\begin{equation}\label{eq:3.0}
{\rm Ric}\geq n-1+\dfrac{n-1}{n}\left(nH^2-\frac{n-2}{\sqrt{n-1}}\sqrt{n}H\sqrt{S-nH^2}-(S-nH^2)\right).
\end{equation}
In order to estimate~\eqref{eq:3.0} from below, let us the following quadratic form with eigenvalues $\pm\dfrac{n}{2\sqrt{n-1}}$:
\begin{equation}\label{eq:3.1}
F(x,y)=x^2-\frac{n-2}{\sqrt{n-1}}xy-y^2.
\end{equation}
We note that the orthogonal transformation
\begin{equation}\label{eq:3.2}
\left\{
\begin{array}{ccc}
w&=\frac{1}{2n}\left[(1+\sqrt{n-1})x+(1-\sqrt{n-1})y\right] \\
v&=\frac{1}{2n}\left[(\sqrt{n-1}-1)x+(1+\sqrt{n-1})y\right]
\end{array}
\right.
\end{equation}
is such that $x^2+y^2=w^2+v^2$. Thus, by taking $x=\sqrt{n}H$ and $y={\sqrt{S-nH^2}}$ we have $x^2+y^2=S$. Hence,
\begin{equation}\label{eq:3.3}
\begin{split}
F(x,y)&=x^2-\frac{n-2}{\sqrt{n-1}}xy-y^2=\frac{n}{2\sqrt{n-1}}(w^2-v^2)\\
&\geq -\frac{n}{2\sqrt{n-1}}(w^2+v^2)\\
&=-\frac{n}{2\sqrt{n-1}}S,
\end{split}
\end{equation}
with equality holding if and only if $v=0$. So, by inserting~\eqref{eq:3.3} in~\eqref{eq:3.0},
\begin{equation}\label{eq:3.4}
{\rm Ric}\geq\left(n-1-\dfrac{\sqrt{n-1}}{2}S\right).
\end{equation}
By replacing~\eqref{eq:3.4} in Lemma~\ref{lem:2.2},
\begin{equation}\label{eq:3.5}
\int_{M}\left(n-1-\dfrac{\sqrt{n-1}}{2}S\right)|\nabla f|^{2p-2}dM\leq\lambda^{\frac{2}{p}}c(n,p)\mathcal{R}(f,z,u)^{\frac{2p-2}{p}}+\int_{\partial M}|\nabla f|^{2p-4}\mathcal{Q}(u,z)d\sigma,
\end{equation}
where $\mathcal{R}(f,z,u)$ and $\mathcal{Q}(u,z)$ are defined in~\eqref{eq:2.12} and~\eqref{eq:2.13}, respectively.

On the other hand, if the boundary $\partial M$ is empty or satisfies the Dirichlet or Neumann boundary condition,
\begin{equation}
\mathcal{R}(f,z,u)=\left(\int_M |\nabla f|^{2p-2}dM\right)^{\frac{p}{2p-2}}\quad\mbox{and}\quad\int_{\partial M}|\nabla f|^{2p-4}\mathcal{Q}(u,z)d\sigma\leq0,
\end{equation}
and thus~\eqref{eq:2.21} becomes
\begin{equation}\label{eq:3.6}
\int_{M}\left(n-1-\dfrac{\sqrt{n-1}}{2}S\right)|\nabla f|^{2p-2}dM\leq\lambda^{\frac{2}{p}}c(n,p)\int_M |\nabla f|^{2p-2}dM.
\end{equation}
Therefore
\begin{equation}\label{eq:3.7}
\int_{M}\left(\dfrac{(n-1)(n-b(n,p)\lambda^{\frac{2}{p}})}{n}-\dfrac{\sqrt{n-1}}{2}S\right)|\nabla f|^{2p-2}dM\leq0,
\end{equation}
where $b(n,p)$ is given by~\eqref{eq:1.8}. This proves the inequality in~\eqref{eq:1.7}.

Moreover, the equality 
\begin{equation}\label{eq:3.8}
\int_{M}\left(\dfrac{(n-1)(n-b(n,p)\lambda^{\frac{2}{p}})}{n}-\dfrac{\sqrt{n-1}}{2}S\right)|\nabla f|^{2p-2}dM=0
\end{equation}
holds, then by Lemma~\ref{lem:2.2} we must have $p=2$. In particular, $b(n,2)=1$ and~\eqref{eq:3.8} turns in
\begin{equation}\label{eq:3.9}
\int_{M}\left(\dfrac{(n-1)(n-\lambda)}{n}-\dfrac{\sqrt{n-1}}{2}S\right)|\nabla f|^{2}dM=0.
\end{equation}
Beyond that, from~\eqref{eq:2.16} we have
\begin{equation}\label{eq:3.10}
\mathrm{Hess}\,f=-\dfrac{\lambda}{n}fI.
\end{equation}
Hence, if $\partial M=\emptyset$, by applying Obata's theorem~\cite[Theorem A]{Obata:62}, it follows that $M^n$ is isometric to the sphere $\mathbb{S}^n(\lambda/n)$.  In the case where the boundary $\partial M$ is nonempty and convex, we can apply~\cite[Theorem 4.3]{Escobar:90} in order to obtain that $M^{n}$ is isometric a hemisphere $\mathbb{S}^{n}_{+}(\lambda/n)$. Since the equality~\eqref{eq:3.8} imply in the equality~\eqref{eq:3.3}, we get $v=0$. Thence,
\begin{equation}\label{eq:3.11}
nH^2=\left(\frac{\sqrt{n-1}-1}{\sqrt{n-1}+1}\right)^2(S-nH^2).
\end{equation}
First we observe that if $n=2$, then $H=0$ and $M^{n}$ is a minimal submanifold of $\mathbb{S}^{N}$. Hence, by using a similar argument to made in~\cite[Theorem 3]{Leung:83} we conclude that $N=2q$ and $M^{n}$ is isometric to $\mathbb{S}^{2}\left(\sqrt{q(q+1)/2}\right)$. 

From now on, we will assume that $n\neq2$. In this case, \eqref{eq:3.11} becomes
\begin{equation}\label{eq:3.12}
S=\dfrac{2n^2}{(\sqrt{n-1}-1)^2}H^2.
\end{equation}
Being $M^n$ is isometric to $\mathbb{S}^n\left(\lambda/n\right)$ (similarly for $\mathbb{S}^{n}_{+}(\lambda/n)$), by putting~\eqref{eq:3.12} in~\eqref{eq:2.7}, we have
\begin{equation}\label{eq:3.13}
\lambda=n+\frac{n^2}{n-1}\left(1-\frac{2}{(\sqrt{n-1}-1)^2}\right)H^2.
\end{equation}
Since $H$ is constant, from identity~\eqref{eq:3.11} we must have that $S$ is constant. So, \eqref{eq:3.13} can be rewrite as follows
\begin{equation}\label{eq:3.14}
\lambda-n=-\dfrac{n}{2\sqrt{n-1}}S.
\end{equation}
Hence, by inserting~\eqref{eq:3.12} and~\eqref{eq:3.14} in~\eqref{eq:3.13} we get
\begin{equation}\label{eq:3.15}
\frac{n^2}{n-1}\left(1-\frac{2}{(\sqrt{n-1}-1)^2}\right)H^2 = -\frac{n}{2\sqrt{n-1}}\frac{2n^2}{(\sqrt{n-1}-1)^2}H^2.
\end{equation}
Thus,
\begin{equation}\label{eq:3.16}
\left(\dfrac{1}{n-1}\left((\sqrt{n-1}-1)^2-{2}\right)+\frac{n}{\sqrt{n-1}}\right)H^{2}=(n-2)(\sqrt{n-1}+1)H^{2}=0.
\end{equation}
Once that $n\neq2$, it follows $H=0$ and hence, $S=0$, from~\eqref{eq:3.12}. Therefore, by returning to~\eqref{eq:3.14}, we conclude that $n=\lambda$ and consequently $M^n$ is isometric to $\mathbb{S}^n(1)$ if $\partial M=\emptyset$ and isometric to $\mathbb{S}^{n}_{+}(1)$ otherwise. 

To ends this proof, we will assume that $\partial M$ is nonempty, convex and satisfy the Dirichlet boundary condition. Since $v=0$, from~\eqref{eq:3.0}
\begin{equation}\label{eq:3.17}
{\rm Ric}\geq n-1+\dfrac{(n-1)(\lambda-n)}{n}=\dfrac{(n-1)\lambda}{n}>0.
\end{equation}
Being $\partial M$ convex, follows that $\mathcal{H}$ is nonpositive. Hence, we can apply the classical result~\cite[Theorem 4]{Reilly:77} in order to obtain that $M^{n}$ is isometric to $\mathbb{S}^{n}_{+}(\lambda/n)$. By thinking as before, we conclude that $M^{n}$ is isometric to $\mathbb{S}^{n}_{+}(1)$.


\section*{Acknowledgment}
The first author is partially supported by CNPq, Brazil, grant 311124/2021-6 and Propesqi (UFPE). The second author is partially supported by CNPq, Brazil.

\end{document}